\documentclass{amsart}
\usepackage{graphicx, amsmath, amsfonts, amssymb, amsthm, hyperref, enumerate, xparse} %
\usepackage{theoremref}
\newtheorem{theorem}{Theorem}
\newtheorem{lemma}[theorem]{Lemma}
\newtheorem{prop}[theorem]{Proposition}

\newcommand{\density}{\rho}
\newcommand{\critdens}{\density_\mathrm{c}}
\newcommand{\Z}{\mathbb{Z}}
\newcommand{\N}{\mathbb{N}}
\newcommand{\E}{\mathbf{E}}
\renewcommand{\P}{\mathbf{P}}
\newcommand{\s}{\mathfrak{s}}
\renewcommand{\j}{\mathfrak{j}}
\newcommand{\psleep}{p_\s}
\newcommand{\pjump}{p_\j}
\newcommand{\config}{\sigma}
\newcommand{\configg}{\tau}
\NewDocumentCommand{\weakconfig}{o}{%
  \config_{\Weak %
  \IfValueTF{#1}{, #1}{}%
  }
}
\NewDocumentCommand{\weakodom}{o}{%
  \odom_{\Weak %
  \IfValueTF{#1}{, #1}{}%
  }
}
\newcommand{\Ch}{\mathsf{Ch}}
\newcommand{\Secondchances}{\mathsf{ACh}}
\NewDocumentCommand{\Returns}{o}{%
  \mathsf{Returns}%
  \IfValueTF{#1}{(#1)}{}%
}

\newcommand{\JumpCount}{J_\origin}
\newcommand{\Weak}{\mathsf{W}}
\newcommand{\Strong}{\mathsf{S}}
\newcommand{\StabSymb}{\mathsf{S}}
\newcommand{\OdomSymb}{\mathsf{O}}
\NewDocumentCommand{\SpacedOp}{ O{} m m }{%
  {#2#3}%
  \IfValueTF{#1}{^{#1}\!}{}%
}
\NewDocumentCommand{\Stab}{o}{\SpacedOp[#1]{}{ \StabSymb }}
\NewDocumentCommand{\WeakStab}{o}{\SpacedOp[#1]{\Weak}{ \StabSymb }}
\NewDocumentCommand{\Odom}{o}{\SpacedOp[#1]{}{ \OdomSymb }}
\NewDocumentCommand{\WeakOdom}{o}{\SpacedOp[#1]{\Weak}{ \OdomSymb }}
\newcommand{\odom}{m}
\newcommand{\strongodom}{\odom_\Strong}
\newcommand{\Jump}{\mathsf{Jump}}
\newcommand{\instr}{\mathcal{I}}
\newcommand{\jumpinstr}[1]{\j_#1}
\newcommand{\origin}{\mathbf{0}}
\DeclareMathOperator{\Ber}{Bernoulli}

\title{Mean-field limit for activated random walk on the integer lattice}

\author{Matthew Junge} \address{Matthew Junge, Department of Mathematics, Baruch College, City University of New York}
	\email{\texttt{matthew.junge@baruch.cuny.edu}}
\author{Harley Kaufman}\address{Harley Kaufman, Department of Mathematics, Baruch College, City University of New York}
	\email{\texttt{Harkauf100@gmail.com} } 
\author{Josh Meisel}\address{Josh Meisel, Department of Mathematics, Graduate Center, City University of New York}
	\email{\texttt{jmeisel@gradcenter.cuny.edu}}

\begin{document}

\begin{abstract}
    We show that the critical density for activated random walk on $\mathbb{Z}^d$ approaches the sleep probability as $d \to \infty$ and provide the first-order correction.
\end{abstract}

\maketitle

\section{Introduction}

Activated random walk (ARW) is a mathematical model of the influential physics theory known as self-organized criticality \cite{bak1987soc, dickman2010activated, rollaSurvey}. Active particles perform independent continuous-time random walks at {jump rate} $1$, and fall asleep at {sleep rate} $\lambda > 0$. Sleeping particles do not move, but become active when visited by another particle. 
 The {fixed-energy} version of ARW on $\Z^d$ starts with a translation-ergodic initial configuration of active particles with mean $\density \in [0,\infty)$. We say the system {fixates} if every particle at some time falls asleep and is not woken up again. If this never occurs, then we say that the system {stays active}.  Treating $\lambda$ as a fixed constant, Rolla, Sidoravicius, and Zindy \cite{universalityTheorem} proved the existence of a {critical density} $\critdens = \critdens(\Z^d, \lambda)$ such that fixation occurs almost surely for $\density < \critdens$ and the system stays active almost surely for $\density > \critdens$.

Let $\psleep := \lambda/(1 + \lambda)$ and $\pjump := 1 - \psleep$ be the sleep and jump probabilities, respectively. One of the first rigorous results for ARW was the lower bound $\critdens \geq \psleep$, proved for $\critdens(\Z, \lambda)$ in \cite{rs} and subsequently generalized to $\critdens(\Z^d, \lambda)$ by Stauffer and Taggi \cite{StaufferTaggi18}.
A mean-field analysis of differential equations that treat the state of each site as statistically independent was conducted by Dickman, Rolla, and Sidoravicius in their seminal work on ARW \cite{dickman2010activated}. They inferred that the critical density in this idealized setting is $\psleep$. This is also the critical density for the special case of directed walks \cite{HS04, rollaSurvey, rt} and simple random walks on regular trees as the degree tends to infinity \cite[Theorem 1.2]{StaufferTaggi18}, \cite[Theorem 1.2]{Taggi2019}. 
J\'arai, M\"onch, and Taggi in \cite{completeGraph} studied ARW on the complete graph on $n+1$ vertices with a single site distinguished as a sink. They proved that the stationary density of sleeping particles left behind after stabilizing the system with one active particle at each non-sink site concentrates around $\psleep + \sqrt{\psleep\pjump\log n /n}$.

It is a hallmark of many statistical physics models that mean-field versions, akin to those in the previous paragraph, approximate the behavior in high dimension where spatial correlations diminish. Indeed, the authors of \cite{completeGraph} conjectured that $\critdens$ ought to be a ``perturbation" of $\psleep$ as $d$ increases. We prove this claim as well as pin down the first order of the perturbation.

\begin{theorem}\thlabel{thm:high-d} 
    For any fixed $\lambda > 0$ as $d \to \infty$
    $$
        \critdens(\Z^d, \lambda) = \psleep + \frac{\psleep\pjump}{2d} + O(d^{-2}).
    $$
\end{theorem}

Our result establishes concordance between mean-field, directed $\mathbb Z^d$, regular tree, complete graph, and $\mathbb Z^d$ ARW. Also, it gives a sharp up to first-order estimate on $\rho_c$ whose exact value for non-directed random walks is unknown. 

\thref{thm:high-d} is an immediate corollary of the following proposition. 
Denote the number of returns to $\origin$ by a random walk started there, not counting the initial occupation, by $\Returns = \Returns(\Z^d) \in \mathbb \{0,1,\hdots ,\infty\}$. 
\begin{prop}\thlabel{prop:ub-lb}
     For any $\lambda > 0$ and $d \ge 1$
    \begin{align}
    \psleep + \frac{\psleep\pjump}{2d}\left(1 - \frac{\psleep\pjump}{2d}\right) \le \critdens(\Z^d, \lambda) \le \psleep + \psleep\pjump \E[\Returns[\Z^d]]. \label{eq:lb}
    \end{align}
\end{prop}
\begin{proof}[Proof of \thref{thm:high-d}]
The lower bound in \thref{prop:ub-lb} is $\psleep + \psleep\pjump (2d)^{-1} + O(d^{-2}).$ The upper bound is also of this order since $\E[\Returns[\Z^d]] = (2d)^{-1} + O(d^{-2})$ (see \cite[Equation A.6]{high-d-rw}).
\end{proof}

\thref{prop:ub-lb} could be extended to obtain a mean-field limit on more general graphs with more general random walk kernels, as well as sharp asymptotics for the critical density as $\lambda \to 0$ on transient graphs. %For example, \thref{prop:ub-lb} gives a shorter proof of \cite[Theorem 1.3]{AsselahAmine2024Tcdf} that $\rho_c = \Theta(\lambda)$ for $d \geq 3$. 
We chose to focus on $\mathbb Z^d$ for the sake of simplicity.  
More details are in Section~\ref{sec:extensions}. 

\subsection*{Organization}

Section~\ref{sec:sketch} sketches the proof of \thref{prop:ub-lb}.
Section~\ref{subsection:site-wise} describes the sitewise representation of ARW as well as some important properties. 
Section~\ref{sec:weak-stab} contains the definitions of weak and strong stabilizations and proofs of a few useful relations.
Section~\ref{sec:proofs} has the proof of \thref{prop:ub-lb}.
Section~\ref{sec:extensions} discusses extensions and further consequences of \thref{prop:ub-lb}.

\section{Proof overview for \thref{prop:ub-lb}} \label{sec:sketch}

We obtain our bounds via improvements to the approaches in  \cite{StaufferTaggi18} and \cite{Taggi2019}. Along with the technical inferences in \thref{lem:fill} and \thref{cor:fill}, our main tool is a simple observation  in \thref{lem:indie-trials} that connects the probability the origin is occupied to the number of chances a particle has to fall asleep there when repeatedly weakly stabilizing. This fact was already known \cite{rollaSurvey}, but it appears that it has not been applied in a precise enough manner to deduce the asymptotics of $\rho_c$. Our main contribution is pairing the formula with sharp enough estimates---\thref{lem:expected-chances} and the proof of the lower bound in \thref{prop:ub-lb}---to pin down the exact first-order asymptotics of $\rho_c$. We give more details below.

\subsection{Lower bound}

To lower bound the critical density, we extend the argument of Stauffer and Taggi \cite{StaufferTaggi18} establishing $\psleep$ as a lower bound. There one stabilizes the configuration $\config$ on a large finite set, say $V_n := \{-n, \ldots, n\}^d \subseteq \Z^d$, by first \emph{weakly stabilizing} $V_n$, that is stabilizing $V_n$ but freezing a single active particle at $\origin$ as soon as one enters. An active particle occupies $\origin$ after weak stabilization if and only if the full stabilizing odometer is positive at $\origin$, so starting with a supercritical configuration, this is guaranteed for large $V_n$. After the weak stabilization then there is a $\psleep$ probability the particle immediately falls asleep, producing a stable configuration $\Stab[V_n]\,\config$ with $\origin$ occupied, which we denote by $\origin \in \Stab[V_n]\,\config$. Then by translation invariance, $\Stab[V_n]\,\config$ has at least density $\psleep$ on sites far from $\partial V_n$, which by amenability implies the supercritical starting density must have been at least $\psleep$. 

We increase this bound by including the probability that the particle at $\origin$ does not fall asleep immediately, but then either returns and falls asleep or wakes a neighboring sleeping particle that returns and falls asleep.
%that if the particle does not fall asleep, instead jumping to an adjacent vertex $x \sim \origin$, it returns to $\origin$ on the next step with probability $\pjump/(2d)$. Then after weakly stabilizing again there is another chance to have $\origin \in \Stab[V_n]\,\config$, bumping the density up to $\psleep + \pjump^2\psleep/(2d)$. Moreover, the particle may have awoken a second particle at $x$, which has a $1/(2d)$ chance of returning, as the first particle keeps it from falling asleep. Arguing similarly as for the $\psleep$ lower bound, this second particle is present with at least probability $\psleep$, producing our lower bound.

\subsection{Upper bound}

We upper bound the density of $\Stab[V]\config$ at $\origin$. This in turn upper bounds the critical density, since at density $\density < \critdens$, due to fixation, $\Stab[V]\,\config$ has density close to $\density$ as $V \nearrow \Z^d$. 
To do so, we continue checking if $\origin \in \Stab[V]\,\config$ via successive weak stabilizations. Each weak stabilization that ends with $\origin$ occupied produces a $\psleep$ chance to have $\origin \in \Stab[V_n]\,\config$, while a weak stabilization that ends with $\origin$ unoccupied completes the stabilization with $\origin \notin \Stab[V]\config$. In \cite{Taggi2019}, the outcomes of the $\Ber(\psleep)$ trials were decoupled from the number of trials. We denote this number of trials by $\Ch$, and refer to it as the number of \emph{chances}. The decoupling is achieved through what Rolla terms \emph{strong stabilization} \cite[Section~7]{rollaSurvey}. In this procedure, the particle is toppled out of $\origin$ regardless of whether it falls asleep. The weak stabilizations are then repeated until one ends with $\origin$ empty, that is, until $V$ is \emph{strongly stable}.

In \cite{Taggi2019} it is proved that $\E[\Ch] \le \E[\Returns + 1]$. This is done by equating strong stabilization with weak stabilization after placing an extra particle at the origin, and then cleverly applying the abelian property and enforced activation, toppling particles in two different orders and comparing the odometers. We modify their argument in \thref{lem:expected-chances} to obtain the slightly stronger statement that the expected number of {additional chances} is at most $\E[\Returns]$. A first-chance success occurs with probability $\psleep$ and the probability of an additional-chance success that follows a first-chance failure is no more than $\pjump \psleep \E[\Returns]$ using Markov's inequality. 
%The other $\psleep$ term in our upper bound handles the possibility of a first-chance success. 

\section{The site-wise representation and some basic properties}\label{subsection:site-wise}

 The definitions and properties introduced in this section are standard. For more details see \cite[Section 2]{rollaSurvey}. In the \emph{site-wise} or Diaconis-Fulton representation of ARW, there are random {instruction stacks} $\instr = (\instr_x(k))_{x \in \Z^d, k \ge 0}$.
The stacks $\instr_x = (\instr_x(k))_{k \ge 0}$ are independent and composed of sleep instructions $\s$ and jump instructions $\jumpinstr{y}$ for $y \sim x$. The stack $\instr_x$ is i.i.d.\ with $\P\left(\instr_x(k) = \s\right) = \psleep$ and $\P\left(\instr_x(k) = \jumpinstr{y}\right) = \pjump/(2d)$ for each $y \sim x$. 

When stabilizing a finite set of vertices $V \subset \Z^d$, killing particles that leave $V$, we represent the state of the system by the {particle \emph{configuration} $\config\colon V \to \N \cup \{\s\}$ and \emph{odometer} $\odom\colon V \to \N$, 
where $\config$ indicates the number of active particles at a site or the presence of a sleeping particle, and $\odom$ counts how many instructions have been used at each site. A site $x \in V$ is {unstable} if it contains at least one active particle. {Toppling} an unstable site $x$ has the effect of updating the configuration according to instruction $\instr_x(\odom(x))$ and then increasing $\odom(x)$ by $1$. The {abelian property} says that if one starts from state $(\config, \odom)$, and stabilizes by toppling unstable sites in any order until every $x \in V$ is stable, the resulting state is always the same. It is composed of 
\[
    \Stab[V](\config, \odom) \in \{0, \s\}^V \text{ and } \Odom[V](\config, \odom) \in \mathbb N ^V,
\]
which we call the \emph{stable configuration} and \emph{stabilizing odometer} for $(\config, \odom)$, respectively.
\begin{prop}[Abelian property]
    Fix finite $V \subset \Z^d$, configuration $\config$ on $V$, and odometer $\odom$. Stabilizing from $(\config, \odom)$ in any order produces stable configuration $\Stab[V](\config, \odom)$ and stabilizing odometer $\Odom[V](\config, \odom)$.
\end{prop}

Typically the odometer starts from zero, so we define $\Stab[V]\config := \Stab[V](\config, 0_V)$ and similarly for $\Odom[V]\config$. Note that we will write $x \in \Stab[V] \config$ to denote $(\Stab[V]\sigma)(x) = \s$, the stabilization leaving a sleeping particle at $x \in V$. We can stabilize configurations $\config\colon W \to \N \cup \{\s\}$ on larger vertex sets $W \supseteq V$ by letting $\Stab[V](\config, \odom) := \Stab[V](\config|_V, \odom)$ and similarly for $\Odom[V](\config, \odom)$.

The \emph{least action principle} states that waking up a sleeping particle via enforced activation only increases the stabilizing odometer. That is, call a site \emph{acceptable} to topple if it is nonempty. Toppling a site with a sleeping particle first wakes up the particle, and then executes the next instruction there. With $\succeq$ denoting pointwise domination, \cite[Lemma 2.1]{rollaSurvey} states that:

\begin{prop}[Least Action Principle]
    For any configuration $\config$ on $V$ and odometer $\odom$, if odometer $\odom'$ is produced by stabilizing from $(\config, \odom)$ via acceptable topplings, then $\odom' \succeq \Odom[V](\config, \odom)$.
\end{prop}

For a configuration $\config\colon \Z^d \to \N \, \cup \, \{\s\}$ on $\Z^d$ and fixed instructions $\instr$, the stabilizing odometer is nondecreasing in $V$. Therefore, we may define the possibly infinite odometer $\odom_{\config}\colon \Z^d \to \N \cup \{\infty\}$ for the system on all of $\Z^d$ by
$$\odom_{\config}  := \lim_{n \to \infty} \Odom[V_n]\,\config.$$
Rolla and Tournier \cite{rt} showed that particle fixation is equivalent to site and odometer fixation in the sense of \thref{prop:site-fixation} below. Recall the defintions of fixation and staying active from the introduction. Below and throughout, we always use a measure where starting configurations $\config$ are independent from the instructions $\instr$.

\begin{prop}[Equivalence of particle and site fixation]\thlabel{prop:site-fixation}
For any translation-ergodic configuration $\config\colon \Z^d \to \N \cup \{\s\}$ it holds that
$$
\P(\config\text{ fixates})=\P(m_{\config}(\origin)<\infty)\in \{0,1\}.
$$
\end{prop}

Therefore, on fixation there is a limiting stable configuration $$\Stab\, \config := \lim_{V \nearrow \Z^d}\Stab[V]\config,$$
since for each $x \in V$ it can be shown that $x \in \Stab[V]$ if and only if the stabilizing odometer ends on a sleep instruction (or no instructions at $x$ are used and $\config(x) = \s$). Under i.i.d.\ conditions $\Stab\,\config$ has the same density as $\config$ \cite{AmirGurel-Gurevich10}.

\begin{prop}[Mass conservation]\thlabel{prop:mass-conservation}
    For any i.i.d.\ initial configuration $\config\colon \Z^d \to \N \cup \{\s\}$ with density $\density = \E \left|\config(\origin)\right|$, if the fixed-energy system starting from $\config$ a.s.\ fixates then 
    $
        \E\left|(\Stab\,\config)(\origin)\right| = \density.
    $
\end{prop}

\section{Weak and strong stabilization}\label{sec:weak-stab}
Take finite $V \subset \Z^d$ and some $U \subseteq V$. Following the definition of the usual stabilization on $V$ from Section~\ref{subsection:site-wise}, which we call the true stabilization when the distinction is needed, we define weak and strong stabilization with respect to $U$. Call $x \in U$ \emph{weakly} (resp., \emph{strongly}) stable with respect to $U$ if $\config(x) \in \{0, \s, 1\}$ (resp., $\config(x)=0$). For $x \in V \setminus U$ the three definitions of stability coincide. \emph{Weakly stabilizing} with respect to $U$ consists of toppling weakly unstable sites until the configuration is weakly stable. This can be viewed as the usual stabilization but with $\lambda=\infty$ on $U$. Define \emph{strongly stabilizing} similarly, where it as if $\lambda=0$ on $U$. 
As noted in \cite{StaufferTaggi18, rollaSurvey}, the abelian property goes through by the same proof in both cases, so there is a well-defined \emph{weakly stable configuration} and \emph{weak stabilizing odometer} \[
    \WeakStab[V,U](\config, \odom), \qquad \WeakOdom[V,U](\config, \odom),
\] as well as the \emph{strong stable configuration} and \emph{strong stabilizing odometer}. The least action principle continues to hold, that is any odometer obtained by acceptably toppling sites until the configuration is weakly (resp. strongly) stable dominates the weak (resp. strong) stabilizing odometer. If the reference to $U$ is omitted, it is assumed that $U = \{\origin\}$, and weakly/strongly stabilizing with respect to $x \in V$ means with respect to $\{x\}$.

We define an iterative three-step procedure for strongly stabilizing an active configuration $\config\colon V \to \N$ with respect to $\origin$. The procedure begins with a single pre-step, weakly stabilizing with respect to the origin once to obtain the configuration $\weakconfig[0]$ and odometer $\weakodom[0]$, which by the abelian property are uniquely determined by $\config, V, \instr$. If $\weakconfig[0](\origin) = 0$ we are done. Otherwise, $\weakconfig[1](\origin) = 1$, and we complete iteration $k=1$:

\begin{enumerate}
    \item (The \emph{jump-out}) Acceptably topple the particle at $\origin$ until it jumps out.
    \item Weakly stabilize with respect to $\origin$, leading to state $(\weakconfig[k], \weakodom[k])$. 
    \item If the origin is empty, stop. Otherwise, repeat Steps 1--3.
\end{enumerate}

Define $\Ch = \Ch(\config, V) \ge 0$ to be the number of iterations completed, and note that $(\weakconfig[\Ch], \weakodom[\Ch])$ is the strongly stable state. Each time the jump-out step is completed, there is a {sleep trial}, which is successful if the particle falls asleep at least once before jumping out. If the first successful sleep trial occurs during iteration $k$, so only unstable sites have been toppled so far, the true stabilization is complete, ending at configuration $\config_k$ but with the particle at the origin put to sleep, so $\origin \in \Stab[V]\config$. If no sleep trial succeeds, the strongly stable state is the true stable state, and $\origin \notin \Stab[V]\config$.

\begin{lemma}\thlabel{lem:indie-trials}
    For any active configuration $\sigma\colon V \to \N$ on a finite $V \subset \Z^d$ containing the origin, \begin{equation}\label{eq:density-by-chances}
        \P(\origin \in \Stab[V]\config) = \sum_{k=1}^\infty\psleep\pjump^{k-1}\P\big(\Ch(\config, V) \ge k\big).
    \end{equation}
\end{lemma}
\begin{proof}
    The equation \eqref{eq:density-by-chances} is explained in the proof of \cite[Proposition 7.12]{rollaSurvey}. Note that $\Ch$ is equivalent to $T_V^s -1$ from \cite{rollaSurvey}. 
\end{proof}

By \eqref{eq:density-by-chances}, the density at the origin is exactly $1 - \E[\pjump^\Ch]$. Thus, the distribution of $\Ch$ holds a lot of information about the critical density. In the proof of the upper bound in \thref{prop:ub-lb}, we will use the following bound on the expected value of $\Secondchances = \Secondchances(\config, V) := \max(\Ch - 1,0)$ the number of \emph{additional chances} to fall asleep at the origin. 

\begin{lemma}\thlabel{lem:expected-chances}
    For any active configuration $\config$ on a finite set $V \subset \Z^d$ containing $\origin$, 
    $$\E\left[\Secondchances(V, \config)\right] \le \E[\Returns(\Z^d)].$$
\end{lemma}

\begin{proof}
    Denote the strong stabilizing odometer by $\strongodom$. For any odometer $\odom$, let $\JumpCount(\odom) = \JumpCount(\odom, \instr)$ count the number of jump instructions to $\origin$ under $\odom$, that is the number of $x \sim \origin$ and $\ell < \odom(x)$ such that $\instr_x(\ell) = \jumpinstr{\origin}$. Then \begin{equation}\label{eq:num-jumps-back-in}
        \Secondchances \le \JumpCount(\strongodom) - \JumpCount(\weakodom[0]),
    \end{equation}
    since for each iteration $1 \le k < \Ch$, after the $k^\text{th}$ jump-out a particle returns to the origin, leading to another iteration.

    Note that strongly stabilizing can be seen as placing an extra particle at the origin and then weakly stabilizing, so $\strongodom$ exactly equals the weak stabilizing odometer for $\config + \delta_\origin$. Consider the following procedure weakly stabilizing $\config+\delta_\origin$ via acceptable topplings: first, acceptably topple the extra particle until it exits $V$, using the odometer $\odom_\mathsf{RW}$ and leading to the original configuration $\config$. Clearly, \begin{equation}\label{eq:rw-exit}
        \JumpCount(\odom_\mathsf{RW}) \preceq_{st} \Returns[\Z^d].
    \end{equation}
    Finish the procedure by weakly stabilizing $(\config, \odom_\mathsf{RW})$, leading to the odometer $$\strongodom' = \odom_\mathsf{RW} + \weakodom'.$$ As the instructions after $\odom_\mathsf{RW}$ remain independent, we have 
    \begin{equation}\label{eq:weak-odoms-equal-in-d}
        \JumpCount(\weakodom') \overset{d}{=} \JumpCount(\weakodom[1]).
    \end{equation} Finally, by the least action principle, \begin{equation}\label{eq:acceptable-domination}
        \strongodom' \succeq \strongodom,
    \end{equation} and therefore, 
    \begin{align*}
        \E[\Secondchances] 
            &\le \E[\JumpCount(\strongodom) - \JumpCount(\weakodom[1])] &&\text{(by \eqref{eq:num-jumps-back-in})}\\
            &\le \E[\JumpCount(\strongodom')] - E[\JumpCount(\weakodom')] &&\text{(by \eqref{eq:acceptable-domination} and \eqref{eq:weak-odoms-equal-in-d})}\\
            &= \E[\JumpCount(\odom_\mathsf{RW})]\\
            &\le \E[\Returns(\Z^d)] &&\text{(by \eqref{eq:rw-exit}).}
    \end{align*}
    
\end{proof}

Finally, we demonstrate the unsurprising fact that a non-fixating configuration can fill any finite set of vertices $U \subset \Z^d$ with an active particle at each vertex. Say that configuration $\configg$ \emph{fills} $x$ if $\configg(x) = 1$, and fills $U$ if it fills every $x \in U$.

\begin{lemma}\thlabel{lem:fill}
    Let $\config$ be a translation-ergodic configuration on $\Z^d$ such that the fixed-energy system starting at $\config$ a.s.\ stays active. Let $V_n \subseteq \mathbb Z^d$ be any non-decreasing sequence of sets (in terms of containment) that converges to $\mathbb Z^d$.
    For any finite $U \subset \Z^d$, the configuration $$\config_{V_n, \Weak} := \WeakStab[V_n,U]\,\config$$ fills $U$ with high probability as $n \to \infty$. 
\end{lemma}

\begin{proof}
    Take $n$ large enough that $V_n \supseteq U$, and let $\odom_{V_n}$ and $\odom_{V_n, \Weak}$
    be the stabilizing and weak stabilizing odometers for $\config$ on $V_n$. Take $x \in U$ and suppose $\config_{V_n, \Weak}$ does not fill $x$, and so $\config_{V_n, \Weak}(x) \in \{0, \s\}$ as it is weakly stable. Then $\odom_{V_n, \Weak}(x)=0$, since $x$ is only toppled during weak stabilization if it has multiple active particles, after which point it always retains at least one and so ends the weak stabilization filled. If then each active particle on $U \setminus \{x\}$ immediately falls asleep, $V_n$ is stabilized with $\odom_{V_n}(x) = 0$. Since there are at most $|U| - 1$ such particles, and the remaining instructions are independent after weak stabilization, \begin{align*}
        \P\big(\odom_{V_n}(x) = 0\big) \ge \P\big(\config_{V_n, \Weak}(x) \neq 1\big)\psleep^{|U| - 1}.
    \end{align*}

    But by \thref{prop:site-fixation}, $\odom_{V_n}(x) \nearrow \infty$ a.s., and so $\odom_{V_n}(x) = 0$ with vanishing probability. Therefore, with high probability \ $\config_{V_n, \Weak}$ fills $x$, as well as all of the finitely many sites in $U$.
\end{proof}

We will need the following consequence of \thref{lem:fill} which will allow us to compare $\critdens$ and $\P(\origin \in \Stab[V] \configg)$ to prove the lower bound in \thref{prop:ub-lb}.

\begin{lemma}\thlabel{cor:fill}
    For any finite $U \subset \Z^d$, $\critdens \ge \density^-_U$ where
    $$
        \density^-_U := \inf_{\substack{\tau \text{ fills } U \\ \text{finite } V \supseteq U}} \P(\origin \in \Stab[V] \configg).
    $$
\end{lemma}
\begin{proof}
    Take some translation-ergodic active configuration $\config\colon \Z^d \to \N$ with density $\density > \critdens$. We will show $\density \ge \density^-_U$. 
    
    For any $V_n \supseteq U$ and finite $V \supseteq  V_n$, begin stabilizing $\config$ on $V$ by attempting to fill $U$, running the weak stabilizing odometer for $U$ on $V_n$. The configuration on $V_n$ after this attempt is then $$\configg_{n} := \WeakStab[V_n, U]\,\config.$$
    Therefore, $$\P(\origin \in \Stab[V] \config) \ge \density^-_U \cdot \P(\configg_{n} \text{ fills } U) =: \density'_U.$$
    For any $N \ge n$ and $x \in V_{N - n}$ then, by translation invariance we have $$\P(x \in \Stab[V_N] \config) \ge \density'_U$$
    and so \begin{equation}\label{eq:density-lb}
        \frac{\E\left|\Stab[V_N] \config\right|}{|V_N|} \ge \frac{\density'_U|V_{N-n}|}{|V_N|}.
    \end{equation}
    On the other hand, since $|\Stab[V_N] \config| \le |\config|$ we have \begin{equation}\label{eq:density-ub}
        \density \ge \frac{\E\left|\Stab[V_N] \config\right|}{|V_N|}.
    \end{equation}
    Combining \eqref{eq:density-lb} and \eqref{eq:density-ub} and taking $N \to \infty$ we get that $$\density \ge \density'_U = \density^-_U \cdot \P(\configg_{n} \text{ fills } U).$$
    Taking $n \to \infty$ then, by \thref{lem:fill} we have $\density \ge \density^-_U$.
\end{proof}

\section{Proof of \thref{prop:ub-lb}}\label{sec:proofs}

\subsection{Upper Bound}
    For the upper bound, fix $\density >  \psleep + \psleep\pjump \E[\Returns[\Z^d]]$ and take some i.i.d.\ active initial configuration $\config$ with mean $\density$. For any finite $V$, 
    \begin{align*}
        \P(\origin \in \Stab[V] \config) 
            &= \sum_{k=1}^\infty \psleep \pjump^{k-1} \P(\Ch \ge k) &&\text{(by \thref{lem:indie-trials})}\\
            &\le \psleep + \psleep\pjump \sum_{k=2}^\infty \P(\Ch \ge k)\\
            &= \psleep + \psleep\pjump \sum_{k=1}^\infty \P(\Secondchances \ge k)\\
            &= \psleep + \psleep\pjump \E[\Secondchances]\\
            &< \density &&\text{(by \thref{lem:expected-chances})}.
    \end{align*}
    By \thref{prop:mass-conservation} then, $\config$ is non-fixating, so $\critdens \le \density$. Since $\density$ is arbitrary, we obtain our claimed upper bound.

\subsection{Lower Bound}
    Let $B_1$ be the ball that includes the origin and its neighbors. Take any finite $B_1 \subseteq V \subset \Z^d$ and configuration $\configg$ that fills $B_1$. We break the strong stabilization of $V$ into five steps. Note that automatically $\Ch = \Ch(\configg, V) \ge 1$, as $\configg(\origin) = 1$.

    \begin{description}
        \item[\textbf{Step 1}] Weakly stabilize $V$ with respect to $\origin$, ending with exactly one active particle at $\origin$.
        \item[\textbf{Step 2}] Jump the particle out of the origin. It visits the uniformly random site $X \in \{x \colon x \sim \origin\}$ which now holds either one or two active particles. 
        \item[\textbf{Step 3}] If there are two particles at $X$, jump one out; otherwise do nothing. 
        \item[\textbf{Step 4}] Topple the lone particle at $X$ once. 
        \item[\textbf{Step 5}] Finish the strong stabilization. 
    \end{description}

Notice that if at either Step 3 or 4 a particle jumps from $X$ to the origin, we will have $\Ch \ge 2$. Let $\Jump_1$ denote this happening during Step 3, and $\Jump_2$ during Step 4, so 
\begin{equation}\label{eq:jump1-or-jump2}
    \P(\Ch \ge 2) \ge \P(\Jump_1 \cup \Jump_2).
\end{equation}

Since Step 4 occurs with independent instructions, 
   \begin{equation} \label{eq:jump-2}
\P(\Jump_2) = \frac{\pjump}{2d}
\quad \text{and} \quad
\Jump_2 \perp\!\!\!\perp \Jump_1.
\end{equation} 
Let $\configg_1$ denote the configuration after Step 1. Since there are two particles to start Step 3 if and only if $\configg_1(X) = \s$ we may infer that
\begin{equation} \label{eq:jump-1}
    \P(\Jump_1) = \frac{\P\big(\configg_1(X) = \s\big)}{2d}.
\end{equation}
To see why \begin{equation} \label{eq:p-two}
    \P\big(\configg_1(X) = \s \big) \ge \psleep,
\end{equation}
note that for each $x \sim \origin$, $$\P\big(\configg_1(x) = \s\big) \ge \psleep,$$
as can be seen by completing Step 1 in the following order: first, weakly stabilize with respect to $\{\origin, x\}$. Next, topple $x$. If doing so executed a sleep instruction, Step 1 ends, with a sleeping particle at $x$. Equation \eqref{eq:p-two} follows since, conditional on $\tau_1$, $X$ is still uniformly distributed. 

    By equations \eqref{eq:jump1-or-jump2}--\eqref{eq:p-two},
    \begin{equation*}
        \P(\Ch \ge 2) \ge \frac{\psleep}{2d} + \frac{\pjump}{2d} - \frac{\psleep}{2d}\frac{\pjump}{2d} = \frac{1}{2d}\left(1 - \frac{\psleep\pjump}{2d}\right).
    \end{equation*}
    Combining this with \thref{lem:indie-trials}, and using that $\Ch \ge 1$ as $\configg$ fills $\origin$, we get
     \begin{align*}
        \P(\origin \in \Stab[V]\configg) \ge \psleep\P(\Ch \ge 1) + \psleep \pjump\P(\Ch \ge 2) \ge  \psleep + \frac{\psleep\pjump}{2d}\left(1 - \frac{\psleep\pjump}{2d}\right).
    \end{align*}
    We are done by \thref{cor:fill}.

\section{Extensions} \label{sec:extensions}

The lower bound of \thref{prop:ub-lb} holds on any infinite vertex-transitive graph $G$ with finite degree, with $2d$ replaced by the degree. For non-amenable graphs, one can use the fact proved in \cite{Taggi2019} that even on a non-amenable vertex-transitive graph, the density cannot increase uniformly away from the boundary after stabilization. Moreover, any translation-invariant jump distribution with finite support is permitted. 

For the upper bound of \thref{prop:ub-lb}, conservation of mass is needed, which has been shown for {unimodular walks}, i.e.\ where the jump distribution is invariant under a transitive unimodular subgroup of $\operatorname{Aut}(G)$ with infinite orbits. Examples include translation-invariant walks on vertex-transitive amenable graphs and simple random walk on $d$-regular trees. Combining this with the known lower bound $\rho_c \geq \psleep$ implies that the critical density is $\Theta(\lambda)$ for transient walks. This was recently established for $\Z^d$ in transient dimensions by Asselah, Forien, and Gaudilli\`ere \cite{AsselahAmine2024Tcdf} via a sophisticated toppling scheme. Our argument is shorter, and holds for all unimodular walks, improving the $O(\sqrt{\lambda})$ bound from \cite{Taggi2019} for amenable graphs. We are able to avoid the amenability assumption by using mass conservation in place of Rolla and Tournier's criterion that ARW stays active if a positive fraction of particles leave a large finite graph \cite{rt}. 
Finally then, \thref{thm:high-d} holds for $d$-regular trees, replacing $2d$ with $d$, as $\E[\Returns(\mathbb{T}^d)] = 1/(d - 2) = d^{-1} + O(d^{-2})$.

\section*{Acknowledgements}
Junge was partially supported by NSF DMS Grant 2238272. Kaufman was partially supported by NSF DMS Grants 2238272 and 2349366. Part of this research was conducted during the 2025 Baruch College Discrete Mathematics NSF Site REU.

\bibliographystyle{amsalpha}
\bibliography{BA.bib}

\end{document}